\documentclass[12pt,leqno]{article}
\usepackage{amsmath,amssymb,amsbsy,amsfonts,amsthm,latexsym,
         amsopn,amstext,amsxtra,euscript,amscd,amsthm,mathdots}

\allowdisplaybreaks

\numberwithin{equation}{section}

\newtheorem{lem}{Lemma}

\newtheorem{prop}{Proposition}

\newtheorem{thm}{Theorem}

\newtheorem{conj}{Conjecture}

\begin{document}

\begin{large}
\centerline{\Large \bf Divisors of an Integer in a Short Interval}
\end{large}
\vskip 10pt
\begin{large}
\centerline{\sc  Patrick Letendre}
\end{large}
\vskip 10pt
\centerline{\it Dedicated to Carl Pomerance on the occasion of his $80th$ birthday.}
\vskip 10pt
\begin{abstract}
Let $\mathcal{D}_{n} \subset \mathbb{N}$ denote the set of the $\tau(n)$ divisors of $n$. We study the function
$$
D_{n}(X,Y):=|\{d \in \mathcal{D}_{n}:\ X \le d \le X+Y\}|
$$
for $Y \le X$.
\end{abstract}
\vskip 10pt
\noindent AMS Subject Classification numbers: 11N37, 11N56, 11N64.

\noindent Key words: divisors, the number of divisors function.

\section{Introduction and notation}

Let $\mathcal{D}_{n} \subset \mathbb{N}$ denote the set of the $\tau(n)$ divisors of $n$. For each integer $n \ge 1$, we consider the function
$$
D_{n}(X,Y):=|\{d \in \mathcal{D}_{n}:\ X \le d \le X+Y\}|.
$$

Some nontrivial estimates are known for $D_{n}(X,Y)$. For example, Lemma 6 from \cite{pl} implies that
$$
D_{n}(X,X) \ll \frac{\tau(n)}{\sqrt{\Omega_{2}(n)}}.
$$
where $\Omega_{2}(n):=\sum_{p^{\beta}\| n}\beta^{2}$. Also, we observe that $D_{n}(X,Y)=D_{n}\bigl(\frac{n}{X+Y},\frac{Yn}{X(X+Y)}\bigr)$. The following conjecture is suggested in the literature, see \cite{pe:mr}, \cite{thc:1} and \cite{thc:2}.

\begin{conj}\label{con:1}
Let $\epsilon > 0$ be fixed. There exists a constant $k_{\epsilon}$ such that, for each integer $n \ge 1$, we have 
$$
D_{n}(n^{1/2},n^{1/2-\epsilon}) \le k_{\epsilon}.
$$
\end{conj}

In this paper, we address the possibility that a really strong general upper bound holds when $Y \ll X^{1-\epsilon}$ for each $\epsilon > 0$. Precisely, we propose the following conjecture.

\begin{conj}\label{con:2}
Let $0 < \theta < 1$ and $0 < \epsilon < \theta$ be fixed. There exists a constant $k_{\epsilon}(\theta)$ such that, for each integer $n \ge 1$, we have 
$$
D_{n}(n^{\theta},n^{\theta-\epsilon}) \le k_{\epsilon}(\theta).
$$
\end{conj}

We will use the functions
$$
\omega(n):=\sum_{p \mid n}1 \quad \mbox{and} \quad V(n):=\max_{p^{\beta} \| n}\beta.
$$

\section{Main results}

The first theorem encapsulates our best result that holds in full generality when $Y \le X^{1-\epsilon}$.

\begin{thm}\label{thm:1}
Let $\eta$ and $\theta$ be fixed real numbers such that $0 < \eta < \theta < 1$. For each integer $n \ge 2$, we have
$$
D_{n}(n^{\theta},n^{\eta}) \ll \tau(n)^{1-\xi(\theta,\eta)}\frac{V(n)\log \tau(n)}{\theta(1-\theta)}
$$
where
\begin{equation}\label{res_1}
\xi(\theta,\eta):=\left\{\begin{array}{ll}
1 & \mbox{if}\ \eta \le \theta^{2}\\
\frac{\theta^{2}}{\eta} & \mbox{if}\ \eta > \theta^{2},\ \theta \le \frac{1}{2}\ \mbox{and}\ 2\eta \le \theta\\
4(\theta-\eta) & \mbox{if}\ \eta > \theta^{2},\ \theta \le \frac{1}{2}\ \mbox{and}\ 2\eta > \theta\\
\frac{(1-\theta)^{2}}{(1-\theta)^{2}+\eta-\theta^{2}} & \mbox{if}\ \eta > \theta^{2},\ \theta > \frac{1}{2}\ \mbox{and}\ 2\eta \le 3\theta-1\\
4(\theta-\eta) & \mbox{if}\ \eta > \theta^{2},\ \theta > \frac{1}{2}\ \mbox{and}\ 2\eta > 3\theta-1.\\
\end{array}\right.
\end{equation}
\end{thm}

A relaxed version of Conjecture \ref{con:2} would involve asking for a larger region in which $\xi(\theta,\eta)=1$ holds, see Proposition \ref{prop:1} below. In the opposite direction, the second theorem is inspired by a simple idea that can be found on the first page of \cite{pe}.

\begin{thm}\label{thm:2}
Let $0 < \theta < 1$ and $0 < \epsilon < \theta$ be fixed real numbers. If Conjecture \ref{con:2} holds, then we have
$$
k_{\epsilon}(\theta) \gg \sqrt{\epsilon}\bigl(\theta(1-\theta)\bigr)^{3/2}\Bigl(\frac{1}{\theta^{\theta}(1-\theta)^{1-\theta}}\Bigr)^{\frac{1}{\epsilon}}.
$$
\end{thm}

\section{Preliminary lemmas}

\begin{lem}\label{lem:1}
Let $d_{1},\dots,d_{k}$ be positive integers. Then, for all $t \in \mathbb{Z}$, we have
$$
[d_{1},\dots,d_{k}]^{\frac{t(t+1)}{2}}\prod_{1 \le i < j \le k}(d_{i},d_{j}) \ge \prod_{1 \le i \le k}d^{t}_{i}
$$
where $[d_{1},\dots,d_{k}]$ denotes the least common multiple of $d_{1},\dots,d_{k}$ and where $(d_{i},d_{j})$ denotes the greatest common divisor of $d_{i}$ and $d_{j}$.
\end{lem}

\begin{proof}
This is Corollaire $1.4$ of \cite{hc}.
\end{proof}

\begin{lem}\label{lem:2}
Let $(x_{i},y_{i})$ for $i=1,\dots,k$ be pairs of positive real numbers that satisfy $\sum_{i=1}^{k}x_{i}\le\sum_{i=1}^{k}y_{i}$. There exists a permutation $\sigma=(\sigma_{1},\dots,\sigma_{k})$ such that for each $1 \le s \le k$,
$$
\sum_{i=1}^{s}x_{\sigma_{i}} \le \sum_{i=1}^{s}y_{\sigma_{i}}.
$$
\end{lem}

\begin{proof}
This is Lemma 7 from \cite{pl}
\end{proof}

\begin{lem}\label{lem:3}
Suppose $\mathcal{M}$ is a set of integers contained in an interval of length $N \ge 1$, and assume that the elements of $\mathcal{M}$ avoid a set $\Omega_{\mathfrak{p}}\neq \mathbb{Z}/\mathfrak{p}\mathbb{Z}$ of congruence classes for every $\mathfrak{p}:=p^{\beta_{p}}\in \mathcal{P}$, a set of powers of distinct primes. Then we have
$$
|\mathcal{M}| \le \frac{N+Q^{2}}{H}
$$
for any $Q \ge 1$, where
$$
H=\sum_{q \le Q}h(q)
$$
and where $h(q)$ is the multiplicative function defined for $\beta \ge 1$ as
$$
h(p^{\beta}):=\left\{\begin{array}{ll}\frac{|\Omega_{\mathfrak{p}}|}{\mathfrak{p}-|\Omega_{\mathfrak{p}}|} & \mbox{if}\ p^{\beta}=\mathfrak{p}\in \mathcal{P},\\
0 & \mbox{otherwise}.
\end{array}\right.
$$
\end{lem}

\begin{proof}
This is essentially the last part of Theorem 7.14 from \cite{hi:ek}, with a slight generalization that allows considering powers of primes instead of only primes. The proof of this generalization is included in Lemma 7.15, assuming that $q=\mathfrak{p}_{1}\cdots\mathfrak{p}_{k}$, and that the $\star$ is now interpreted as $\mathfrak{p}_{i}\nmid a$ for $i=1,\dots,k$.
\end{proof}

In the next lemma, we use the Legendre symbol $\bigl(\frac{\cdot}{p}\bigr)$ for $p \ge 3$.

\begin{lem}\label{lem:4}
Let $p \ge 3$ be a prime number and $u \in \mathbb{Z}$ an integer not divisible by $p$. Let $S_{v}$ be the number of $x$ modulo $p^{v+1}$ for which the equation
$$
x^{2}+p^{v}u \equiv y^{2} \pmod{p^{v+1}}
$$
is unsolvable. Then $S_{0} = \frac{p-(\frac{-u}{p})}{2}$, $S_{1} = p$ and $S_{v} = pS_{v-2}$ for $v \ge 2$.
\end{lem}

\begin{proof}
For $v = 0$, we consider the function
$$
g(k):=\sum_{x=1}^{p}\Bigl(\frac{x^{2}+k}{p}\Bigr).
$$
From here, using changes of variable of the form $x=cz$ with $c \not\equiv 0 \pmod{p}$, we deduce that $g(k)$ only takes two values when $p \nmid k$ depending on the value of $\bigl(\frac{k}{p}\bigr)$. Thus, by summing over $k$ the functions $g(k)$ and $\bigl(\frac{k}{p}\bigr)g(k)$, we arrive at a linear system of two equations with two variables that allows us to evaluate $g(k)$ exactly. We find that $g(k)=-1$ for each $k \not\equiv 0 \pmod{p}$ and we conclude that $S_{0} = \frac{p-(\frac{-u}{p})}{2}$.

Now, for a fixed $v \ge 1$, we consider the change of variables defined by $x = ip^{v} + j$ and $y = rp^{v} + s$ where $i, r = 1, \dots, p$ and $j, s = 1, \dots, p^{v}$. We then treat $x$ as fixed and not divisible by $p$, that means with $i, j$ fixed and $p \nmid j$. We verify that the $y$ defined by $s = j$ and $r \equiv i + (2j)^{-1}u \pmod{p}$ satisfies the equation, and thus these $x$ do not contribute to $S_{v}$.

Finally, again for a fixed $v \ge 1$, we consider the change of variables $x = pX$ and $y = pY$ (since clearly $x \equiv \pm y \pmod{p}$). We then obtain the equation
$$
p^{2}X^{2}+p^{v}u \equiv p^{2}Y^{2} \pmod{p^{v+1}}.
$$
We deduce that $S_{1} = p$, and for $v \ge 2$, we can divide the equation by $p^{2}$. After doing so, we consider the change of variables defined by $X = ip^{v-1} + j$ and $Y = rp^{v-1} + s$ where $i, r = 1, \dots, p$ and $j, s = 1, \dots, p^{v-1}$. We conclude that $S_{v} = pS_{v-2}$, and the result is proven.
\end{proof}

\section{Proof of Theorem \ref{thm:1}}

We first need to establish three very specific statements.

\begin{prop}\label{prop:1}
Let $n \ge 1$ be a fixed integer, and let $0 < \theta < 1$ and $0 < \epsilon < \theta^{2}$ be fixed real numbers. Then
$$
D_{n}(n^{\theta},n^{\theta^{2}-\epsilon}) \ll \frac{\theta(1-\theta)}{\epsilon}+\frac{1}{\theta(1-\theta)}.
$$
\end{prop}

\begin{proof}
We can assume that we have $k \ge 2$ such divisors. From Lemma \ref{lem:1}, we derive
\begin{eqnarray*}
n^{\frac{t(t+1)}{2}}n^{(\theta^{2}-\epsilon)\frac{k(k-1)}{2}} & \ge & [d_{1},\dots,d_{k}]^{\frac{t(t+1)}{2}}\prod_{1 \le i < j \le k}(d_{i},d_{j})\\
& \ge & \prod_{1 \le i \le k}d^{t}_{i}\\
& \ge & n^{\theta tk}.
\end{eqnarray*}
From here, we set $t = \theta k+\zeta$ for some $0 \le \zeta < 1$. By comparing the exponents, we get to the inequality
$$
\epsilon k(k-1) \le (\theta-\theta^{2})k+\zeta(\zeta+1).
$$
Thus, if $k \le \frac{4}{\theta(1-\theta)}$ does not hold, we find $k \le 3\frac{\theta(1-\theta)}{\epsilon}$, and the result follows.
\end{proof}

\begin{prop}\label{prop:2}
Let $\eta$ and $\theta$ be fixed real numbers such that $0 < \theta^{2} \le \eta < \theta < 1$. The largest $\alpha(\theta,\eta)$ for which
\begin{equation}\label{pr_1}
\Bigl(\frac{\theta-t}{\alpha}\Bigr)^{2}-\Bigl(\frac{\eta-t}{\alpha}\Bigr) \ge 0 \quad \mbox{for all} \quad 0 \le t \le 1-\alpha
\end{equation}
holds for $0 < \alpha \le \alpha(\theta,\eta)$ is given by
\begin{eqnarray}\label{tab_1}
\alpha(\theta,\eta):=\left\{\begin{array}{ll}
\frac{\theta^{2}}{\eta} & \mbox{if}\ \theta \le \frac{1}{2}\ \mbox{and}\ 2\eta \le \theta\\
4(\theta-\eta) & \mbox{if}\ \theta \le \frac{1}{2}\ \mbox{and}\ 2\eta > \theta\\
\frac{(1-\theta)^{2}}{(1-\theta)^{2}+\eta-\theta^{2}} & \mbox{if}\ \theta > \frac{1}{2}\ \mbox{and}\ 2\eta \le 3\theta-1\\
4(\theta-\eta) & \mbox{if}\ \theta > \frac{1}{2}\ \mbox{and}\ 2\eta > 3\theta-1.\\
\end{array}\right.
\end{eqnarray}
\end{prop}

\begin{proof}
Given an $\alpha > 0$, we define $F(t):=\bigl(\frac{\theta-t}{\alpha}\bigr)^{2}-\bigl(\frac{\eta-t}{\alpha}\bigr)$ and $I:=[0,1-\alpha]$. $F(t)$ is a polynomial of degree two with zeros
$$
\theta-\frac{\alpha}{2} \pm \frac{\sqrt{\alpha(\alpha-\Delta)}}{2}\quad \mbox{where}\quad \Delta=4(\theta-\eta).
$$
If this polynomial is positive on the interval $I$, it means that its two zeros are either to the left or to the right of $I$, or that $\alpha \le \Delta$. The choice $\alpha=\Delta$ always has the desired property \eqref{pr_1}, so $\alpha > 0$ is possible.

Now, assume that $F$ has a real zero. There are two possible situations. First,
$$
\theta-\frac{\alpha}{2} + \frac{\sqrt{\alpha(\alpha-\Delta)}}{2}\le 0
$$
which gives
\begin{equation}\label{case_1}
\alpha \le \frac{\theta^{2}}{\eta}\quad \mbox{if}\quad \alpha \ge 2\theta\ \mbox{and}\ \alpha \ge \Delta
\end{equation}
and second,
$$
\theta-\frac{\alpha}{2} - \frac{\sqrt{\alpha(\alpha-\Delta)}}{2} \ge 1-\alpha
$$
which gives
\begin{equation}\label{case_2}
\alpha \le \frac{(1-\theta)^{2}}{(1-\theta)^{2}+\eta-\theta^{2}}\quad \mbox{if}\quad \alpha \ge 2(1-\theta)\ \mbox{and}\ \alpha \ge \Delta.
\end{equation}
We remark that in all possible scenarios, $\alpha \le \Delta$, $\alpha \le \frac{\theta^{2}}{\eta}$, and $\alpha \le \frac{(1-\theta)^{2}}{(1-\theta)^{2}+\eta-\theta^{2}}$, we have $\alpha \le 1$.

The first case \eqref{case_1} is possible when $\theta \le \frac{1}{2}$. We verify that $\frac{\theta^{2}}{\eta}-\Delta = \frac{(\theta-2\eta)^{2}}{\eta}$. The condition $\alpha \ge 2\theta$ allows us to complete the first two lines of \eqref{tab_1}. Similarly, the second case \eqref{case_2} is possible when $\theta \ge \frac{1}{2}$ and we then verify that $\frac{(1-\theta)^{2}}{(1-\theta)^{2}+\eta-\theta^{2}}-\Delta = \frac{(2\eta-3\theta+1)^{2}}{(1-\theta)^{2}+\eta-\theta^{2}}$. Again, the condition $\alpha \ge 2(1-\theta)$ allows us to complete the two last lines of \eqref{tab_1}. The proof is complete.
\end{proof}

\begin{prop}\label{prop:3}
Let $\eta$ and $\theta$ be fixed real numbers such that
$$
0 < \theta^{2}-\bigl(\theta(1-\theta)\bigr)^{2}\delta \le \eta < \theta < 1
$$
for some fixed $\delta \le 1$. We can choose $\alpha(\theta,\eta,\delta)$ such that the property
\begin{equation}\label{pr_2}
\Bigl(\frac{\theta-t}{\alpha}\Bigr)^{2}-\Bigl(\frac{\eta-t}{\alpha}\Bigr) \ge \bigl(\theta(1-\theta)\bigr)^{2}\delta \quad \mbox{for all} \quad 0 \le t \le 1-\alpha
\end{equation}
holds for $0 < \alpha \le \alpha(\theta,\eta,\delta)$ with
\begin{eqnarray}\label{tab_2}
\xi(\theta,\eta)-\delta \le \alpha(\theta,\eta,\delta) \le \xi(\theta,\eta).
\end{eqnarray}
\end{prop}

\begin{proof}
For simplicity, we define $\varepsilon:=\bigl(\theta(1-\theta)\bigr)^{2}\delta$. The setup being comparable to that of Proposition \ref{prop:2}, the initial treatment is the same. We define the function $F_{\varepsilon}(t):=F(t)-\varepsilon$ and the same interval $I:=[0,1-\alpha]$. We then consider the zeros of $F_{\varepsilon}(t)$ given by
$$
\theta-\frac{\alpha}{2} \pm \frac{\sqrt{\alpha((1+4\varepsilon)\alpha-\Delta)}}{2}\quad \text{where}\quad \Delta=4(\theta-\eta).
$$
We distinguish three scenarios. In the first case, we consider the situation where $F_{\varepsilon}(t)$ has no roots or just one double root. We find $0 < \alpha \le \frac{\Delta}{1+4\varepsilon}$. The second case is that of two roots less than or equal to $0$. This situation is described by
\begin{equation}\label{c_1}
\alpha \le \frac{\theta^{2}}{\eta+\varepsilon\alpha}\quad \text{if}\quad \alpha \ge 2\theta\ \text{and}\ \alpha \ge \frac{\Delta}{1+4\varepsilon}.
\end{equation}
We can express the solution to the first inequality as $0 < \alpha \le \frac{-\eta+\sqrt{\eta^{2}+4\varepsilon\theta^{2}}}{2\varepsilon}=:\alpha_{0}(\theta,\eta,\delta)$. The third case is that of two roots greater than or equal to $1-\alpha$, a situation that is described by
\begin{equation}\label{c_2}
\alpha \le \frac{(1-\theta)^{2}}{(1-\theta)^{2}+\eta-\theta^{2}+\varepsilon\alpha}\quad \text{if}\quad \alpha \ge 2(1-\theta)\ \text{and}\ \alpha \ge \frac{\Delta}{1+4\varepsilon}.
\end{equation}
The explicit solution is written as $0 < \alpha \le \frac{-(1-2\theta+\eta)+\sqrt{(1-2\theta+\eta)^{2}+4\varepsilon(1-\theta)^{2}}}{2\varepsilon}=:\alpha_{1}(\theta,\eta,\delta)$, we also note that $1-2\theta+\eta > 0$. With these descriptions in hand, we can address inequality \eqref{tab_2}. First, let us verify that each valid $\alpha$ is less than or equal to 1. For $\eta \ge \theta^{2}$, this is obvious, but for $\theta^{2}-\bigl(\theta(1-\theta)\bigr)^{2}\delta \le \eta < \theta^{2}$, we must check the three possibilities separately. The case $\frac{\Delta}{1+4\varepsilon}$ is straightforward. For case \eqref{c_1}, if this were not true, we would have $\varepsilon+\eta < \varepsilon\alpha^{2}+\eta\alpha \le \theta^{2}$, which is a contradiction. The case \eqref{c_2} follows in the same way.

The condition $\alpha \ge 2\theta$ is necessary for the first inequality in \eqref{c_1} to hold, but this is not the case for $\alpha \ge \frac{\Delta}{1+4\varepsilon}$. The same is true in \eqref{c_2}. Also, we note that $2\theta \le \frac{\Delta}{1+4\epsilon}$ for $2\eta \le \theta-4\theta\varepsilon$ and that $2(1-\theta) \le \frac{\Delta}{1+4\varepsilon}$ for $2\eta \le 3\theta-1-4(1-\theta)\varepsilon$. We deduce that if $2\eta \le \theta-4\theta\varepsilon$ and $\alpha_{0}(\theta,\eta,\delta) \ge \frac{\Delta}{1+4\varepsilon}$, then \eqref{pr_2} holds for $0 < \alpha \le \alpha_{0}(\theta,\eta,\delta)$. Similarly, if $2\eta \le 3\theta-1-4(1-\theta)\varepsilon$ and $\alpha_{1}(\theta,\eta,\delta) \ge \frac{\Delta}{1+4\varepsilon}$, then \eqref{pr_2} holds for $0 < \alpha \le \alpha_{1}(\theta,\eta,\delta)$.

Now, we need to address the five cases of \eqref{res_1} separately. Let us begin with cases 3 and 5, that is, when $\xi(\theta,\eta)=\Delta$. Since the property \eqref{pr_2} holds for every $0 < \alpha \le \alpha(\theta,\eta,\delta):=\frac{\Delta}{1+4\varepsilon}$, we have $\xi(\theta,\eta)-\alpha(\theta,\eta,\delta) \le 16\bigl(\theta(1-\theta)\bigr)^{3}\delta \le \frac{\delta}{4}$.

Suppose now we are in case 2 with $\xi(\theta,\eta)=\frac{\theta^{2}}{\eta}$. We distinguish two situations. The first one is when $2\eta \le \theta-4\theta\varepsilon$. We then choose $\alpha(\theta,\eta,\delta):=\alpha_{0}(\theta,\eta,\delta)$. We deduce from above that property \eqref{pr_2} holds for every $0 < \alpha \le \alpha(\theta,\eta,\delta)$ and we have
\begin{eqnarray*}
\xi(\theta,\eta)-\alpha(\theta,\eta,\delta) & = & \frac{\theta^{2}}{\eta}-\frac{\theta^{2}}{\eta+\varepsilon\alpha_{0}(\theta,\eta,\delta)}\\
& \le & \frac{\varepsilon\theta^{2}}{\eta^{2}} \le (1-\theta)^{2}\delta \le \delta
\end{eqnarray*}
since $\alpha_{0}(\theta,\eta,\delta) \le 1$.

The second possibility is when $\theta-4\theta\varepsilon < 2\eta \le \theta$ where we choose $\alpha(\theta,\eta,\delta):=\frac{\Delta}{1+4\varepsilon}$. In this situation, we have property \eqref{pr_2} that holds for every $0 < \alpha \le \alpha(\theta,\eta,\delta)$ and also
\begin{eqnarray*}
\xi(\theta,\eta)-\alpha(\theta,\eta,\delta) & = & \xi(\theta,\eta)-\Delta+\Delta-\frac{\Delta}{1+4\varepsilon}\\
& = & \frac{(\theta-2\eta)^{2}}{\eta}+\frac{4\varepsilon\Delta}{1+4\varepsilon}\\
& \le & \bigl(16\bigl(\theta(1-\theta)\bigr)^{4}+16\bigl(\theta(1-\theta)\bigr)^{3}\bigr)\delta\\
& \le & \frac{5\delta}{16}
\end{eqnarray*}
as required. Case 4 is treated similarly.

It only remains to handle case 1. We will start by considering the case where $\theta \le \frac{1}{2}$, which corresponds to either $\alpha_{0}(\theta,\eta,\delta)$ or $\frac{\Delta}{1+4\varepsilon}$. We proceed as previously and we first consider when $2\eta \le \theta-4\theta\varepsilon$ where we set $\alpha(\theta,\eta,\delta):=\alpha_{0}(\theta,\eta,\delta)$. Thus, property \eqref{pr_2} holds for every $0 < \alpha \le \alpha(\theta,\eta,\delta)$. Also, we already know that $\alpha_{0}(\theta,\eta,\delta) \le 1$ and we have
\begin{eqnarray*}
\xi(\theta,\eta)-\alpha(\theta,\eta,\delta) & = & 1-\frac{\theta^{2}}{\eta+\varepsilon\alpha_{0}(\theta,\eta,\delta)}\\
& = & \frac{\eta+\varepsilon\alpha_{0}(\theta,\eta,\delta)-\theta^{2}}{\eta+\varepsilon\alpha_{0}(\theta,\eta,\delta)}\\
& \le & \frac{\varepsilon}{\theta^{2}} = (1-\theta)^{2}\delta \le \delta
\end{eqnarray*}
as desired. Now, in the situation where $2\eta > \theta-4\theta\varepsilon$, we choose $\alpha(\theta,\eta,\delta) := \frac{\Delta}{1+4\varepsilon}$. From $\eta \le \theta^{2}$, we deduce that $\theta \ge \frac{1-4\varepsilon}{2}$, so that 
\begin{eqnarray*}
\alpha(\theta,\eta,\delta) & \ge & \frac{4(\theta-\theta^{2})}{1+4\varepsilon}\\
& \ge & 4\Bigl(\frac{1-4\varepsilon}{2}-\Bigl(\frac{1-4\varepsilon}{2}\Bigr)^{2}\Bigr)/(1+4\varepsilon)\\
& = & 1-4\varepsilon\\
& \ge & 1-\frac{\delta}{4}
\end{eqnarray*}
which gives the result. The case $\theta > \frac{1}{2}$ is treated in a similar way. The proof is complete.
\end{proof}

We are now ready to begin the proof of the theorem. We can assume that $\tau(n) \ge 3$ since otherwise the result is trivial. We will distinguish two cases. The first case is when $\eta < \theta^{2}-\frac{(\theta(1-\theta))^{2}}{\log \tau(n)}$. By Proposition \ref{prop:1}, we have
\begin{eqnarray*}
D_{n}(n^{\theta},n^{\eta}) & \ll & \frac{\log \tau(n)}{\theta(1-\theta)}\\
& \ll & \tau(n)^{1-\xi(\theta,\eta)}\frac{V(n)\log \tau(n)}{\theta(1-\theta)}
\end{eqnarray*}
which is the desired result. The second case is when $\eta \ge \theta^{2}-\frac{(\theta(1-\theta))^{2}}{\log \tau(n)}$.

Let $n=p_{1}^{\beta_{1}}\cdots p_{\omega(n)}^{\beta_{\omega(n)}}$ be the factorization of $n$ into distinct primes and suppose that $n=ab$ where $(a,b)=1$. We notice the identity
\begin{equation}\label{D_id}
D_{n}(X,Y) = \sum_{e | b}D_{a}\Bigl(\frac{X}{e},\frac{Y}{e}\Bigr).
\end{equation}

To choose the factorization, we begin by setting $\delta := \frac{1}{\log \tau(n)}$. Let us remark right away that we can assume that $\delta < \xi(\theta,\eta)$, since otherwise the result is trivial. Then, we consider the pairs $(\rho_{i},\theta_{i})$, $i=1,\dots,\omega(n)$, defined by $(\beta_{i}+1)=\tau(n)^{\rho_{i}}$ and $p_{i}^{\beta_{i}}=n^{\theta_{i}}$. From Lemma \ref{lem:2}, there exists a permutation $\sigma$ such that
$$
\sum_{i=1}^{s}\rho_{\sigma_{i}} \le \sum_{i=1}^{s}\theta_{\sigma_{i}}
$$
for every $1 \le s \le \omega(n)$. We consider $\sigma$ fixed and choose the minimal $s$ such that
$$
\sum_{i=1}^{s}\theta_{\sigma_{i}} \ge 1-\alpha(\theta,\eta,\delta), \quad \mbox{thus} \quad \sum_{i=1}^{s-1}\theta_{\sigma_{i}} < 1-\alpha(\theta,\eta,\delta).
$$
We define
$$
b:=\prod_{i=1}^{s}p_{\sigma_{i}}^{\beta_{\sigma_{i}}}.
$$
With this factorization $n=ab$ established, we observe that
\begin{eqnarray}\nonumber
\tau(b) & = & (\beta_{\sigma_{s}}+1) \cdot \prod_{i=1}^{s-1}(\beta_{\sigma_{i}}+1)\\ \nonumber
& \le & 2V(n) \tau(n)^{1-\alpha(\theta,\eta,\delta)}\\ \label{tau_b}
& \ll & V(n) \tau(n)^{1-\xi(\theta,\eta)}
\end{eqnarray}
from Proposition \ref{prop:3}. Also, since $a=n/b$, we have $1 \le a = n^{\alpha} \le n^{\alpha(\theta,\eta,\delta)}$. If $a = 1$, the result follows from \eqref{D_id}, \eqref{tau_b} and the fact that $D_{1}(\cdot,\cdot) \le 1$. We can now assume that $a> 1$. We seek to bound $D_{a}\bigl(\frac{X}{e},\frac{Y}{e}\bigr)$ for a given divisor $e$ of $b$, which we note $e=n^{t}$. Let us remark that $D_{a}\bigl(\frac{X}{e},\frac{Y}{e}\bigr) \le 1$ if either $e > Y$ or $\frac{X}{e} > \frac{a}{2}$. Also, by Proposition \ref{prop:3}, we know that
\begin{eqnarray*}
\Bigl(\frac{\log (X/e)}{\log a}\Bigr)^{2}-\Bigl(\frac{\log (Y/e)}{\log a}\Bigr) & = & \Bigl(\frac{\theta-t}{\alpha}\Bigr)^{2}-\Bigl(\frac{\eta-t}{\alpha}\Bigr)\\
& \ge & \frac{\bigl(\theta(1-\theta)\bigr)^{2}}{\log \tau(n)}
\end{eqnarray*}
for $t \in [0,\alpha]$. In any case, from Proposition \ref{prop:1} and \eqref{D_id}, we deduce that
\begin{eqnarray*}
D_{n}(X,Y) & \ll & V(n) \tau(n)^{1-\xi(\theta,\eta)}\Bigl(\frac{\log \tau(n)}{\theta(1-\theta)}+\frac{1}{\theta(1-\theta)}\Bigr)\\
& \ll & \tau(n)^{1-\xi(\theta,\eta)}\frac{V(n)\log \tau(n)}{\theta(1-\theta)}.
\end{eqnarray*}
The proof is complete.

\section{Proof of Theorem \ref{thm:2}}

We need the following proposition to deal with certain cases.

\begin{prop}\label{prop:4}
Let $0 < \theta < 1$ and $0 < \epsilon < \theta$ be fixed real numbers. If Conjecture \ref{con:2} holds, then
$$
k_{\epsilon}(\theta) \ge 1.
$$
\end{prop}

\begin{proof}
Let us begin with the case $\theta \le \frac{1}{2}$. Fix a sufficiently large integer $m$, and the smallest integer $v$ such that $m < v^{\theta}$. There is a multiple of $m$ in the interval $[v-m, v-1]$, say $n_{0}$. We then have
\begin{eqnarray*}
n_{0}^{\theta} & \le & m \le (n_{0}+m)^{\theta}\\
& \le & n_{0}^{\theta}\Bigl(1+\frac{m}{n_{0}}\Bigr)^{\theta}\\
& \le & n_{0}^{\theta}\Bigl(1+\theta\frac{m}{n_{0}}\Bigr)\\
& \le & n_{0}^{\theta}+\frac{1.1\theta}{n_{0}^{1-2\theta}},
\end{eqnarray*}
which is more than sufficient to show that $k_{\epsilon}(\theta) \ge 1$ for $0 < \theta \le \frac{1}{2}$ and $0 < \epsilon < \theta$. Now, consider the case where $\theta > \frac{1}{2}$. A similar argument allows us to find an integer $n_{1}$ having a divisor $m$ such that
$$
n_{1}^{1-\theta}-\frac{1.1(1-\theta)}{n_{1}^{2\theta-1}} \le m \le n_{1}^{1-\theta}
$$
so
$$
n_{1}^{\theta} \le \frac{n_{1}}{m} \le \frac{n_{1}^{\theta}}{1-\frac{1.1(1-\theta)}{n_{1}^{\theta}}} \le n_{1}^{\theta}+1.2(1-\theta)
$$
from which the result follows.
\end{proof}

The parameters $\theta$ and $\epsilon$ are fixed throughout the argument. Let $M$ be a sufficiently large real number for these parameters, and let $L = (\log M)^{2}$. By the Prime Number Theorem, we can assume that
$$
\pi(M+L)-\pi(M) \ge s
$$
where $s \ge 1$ is an integer to be determined later. Define the integer
$$
N_{0}:=p_{1}\cdots p_{s},
$$
where $p_i$ are the $s$ smallest primes greater than $M$. Note that for any group of $\ell \ge 1$ integers $n_{j}$ in the interval $[M, M+L]$, we have
$$
M^{\ell} \le \prod_{j=1}^{\ell}n_{j} \le (M+L)^{\ell}
$$
so that
\begin{equation}\label{int_size}
\prod_{j=1}^{\ell}n_{j}=M^{\ell}+\lambda \ell L(M+L)^{\ell-1}
\end{equation}
for some $0 \le \lambda \le 1$.

Let us begin with the case $\theta \le \frac{1}{2}$. Denoting the ceiling function by $\lceil\cdot\rceil$ and the fractional part by $\{\cdot\}$, we will choose an integer $r \le s$ defined by
$$
r:=\left\{\begin{array}{cl} \lceil\theta s\rceil & \mbox{if}\ 0 < \{\theta s\} \le \frac{1}{2},\\
\lceil\theta s\rceil+1 & \mbox{otherwise}
\end{array}\right.
$$
and also
$$
N:=\max_{\substack{v \ge 1 \\ M^{r} \ge v^{\theta}M^{\theta s}\bigl(1+\theta s L\frac{(M+L)^{s-1}}{M^{s}}\bigr)}}vN_{0},
$$
which in turn defines an integer $m$ such that $N=mN_{0}$. Now, let $\mathcal{Q}$ be the set of the $\binom{s}{r}$ divisors of $N_{0}$ formed by $r$ primes. We will now fix a $q \in \mathcal{Q}$ and show that $N^{\theta} \le q \le N^{\theta-\epsilon}$ if $\epsilon \le \frac{1}{s+3/2\theta}$. On the one hand, we have
\begin{eqnarray}\nonumber
N^{\theta} & \le & m^{\theta}(M^{s}+sL(M+L)^{s-1})^{\theta}\\ \nonumber
& = & m^{\theta}M^{\theta s}\Bigl(1+\frac{sL(M+L)^{s-1}}{M^{s}}\Bigr)^{\theta}\\ \label{ineg_1}
& \le & m^{\theta}M^{\theta s}\Bigl(1+\frac{\theta sL(M+L)^{s-1}}{M^{s}}\Bigr),
\end{eqnarray}
from which we obtain $q \ge N^{\theta}$ by the choice of $N$ and \eqref{int_size}. On the other hand, since
$$
\Bigl(1+\frac{1}{m}\Bigr)^{\theta} \le 1+\frac{\theta}{m} \le 1+\frac{1}{M}
$$
by the choice of $m$, we find
\begin{eqnarray}\nonumber
q & \le & M^{r}\Bigl(1+\frac{rL(M+L)^{r-1}}{M^{r}}\Bigr)\\ \nonumber
& \le & (m+1)^{\theta}M^{\theta s}\Bigl(1+\theta s L\frac{(M+L)^{s-1}}{M^{s}}\Bigr)\Bigl(1+\frac{rL(M+L)^{r-1}}{M^{r}}\Bigr)\\ \nonumber
& \le & \frac{(m+1)^{\theta}}{m^{\theta}}N^{\theta}\Bigl(1+\theta s L\frac{(M+L)^{s-1}}{M^{s}}\Bigr)\Bigl(1+\frac{rL(M+L)^{r-1}}{M^{r}}\Bigr)\\ \nonumber
& \le & N^{\theta}\Bigl(1+\frac{1}{M}\Bigr)\Bigl(1+\theta s L\frac{(M+L)^{s-1}}{M^{s}}\Bigr)\Bigl(1+\frac{rL(M+L)^{r-1}}{M^{r}}\Bigr)\\ \nonumber
& \le & N^{\theta}\Bigl(1+O\Bigl(\frac{sL}{M}\Bigr)\Bigr)\\ \label{ineg_2}
& \le & N^{\theta}+N^{\theta-\epsilon}
\end{eqnarray}
where the last line follows from $\frac{r}{\theta} < s+\frac{3}{2\theta}$.

It remains to show that $\binom{s}{r}$ satisfies the announced inequality. We begin by choosing $s:=\lfloor\frac{1}{\epsilon}-\frac{3}{2\theta}\rfloor$, in accordance with the previous steps, and we will temporarily assume that $\epsilon \le \frac{\theta}{2}$. Note that this assumption implies that $1 \le \frac{1}{2\theta} \le \frac{1}{\epsilon}-\frac{3}{2\theta}$, and therefore $\theta s \ge \frac{1}{4}$ (since $\lfloor x\rfloor \ge \frac{x}{2}$ for $x \ge 1$). We then deduce that $r \asymp \theta s$, and similarly $s-r+1 \asymp s$. We also deduce that $\frac{1}{\epsilon} \ge \frac{1}{\epsilon}-\frac{3}{2\theta} \ge \frac{1}{4\epsilon} \ge 1$, and hence $s \asymp \frac{1}{\epsilon}$.

Now, using the Stirling formula, we have
\begin{eqnarray*}
\binom{s}{r} & \gg & \sqrt{\frac{s}{r(s-r+1)}}\frac{s^{s}}{r^{r}(s-r)^{s-r}}\\
& = & \sqrt{\frac{s}{r(s-r+1)}}\Bigl(\frac{1}{\frac{r}{s}^{\frac{r}{s}}(1-\frac{r}{s})^{1-\frac{r}{s}}}\Bigr)^{s}\\
& \gg & \sqrt{\frac{s}{r(s-r+1)}}\Bigl(\frac{1}{\theta^{\theta}(1-\theta)^{1-\theta}}\Bigr)^{s}\\
& \gg & \sqrt{\frac{\epsilon}{\theta}}\Bigl(\frac{1}{\theta^{\theta}(1-\theta)^{1-\theta}}\Bigr)^{\frac{1}{\epsilon}-\frac{3}{2\theta}}\\
& \gg & \sqrt{\epsilon\theta^{3}}\Bigl(\frac{1}{\theta^{\theta}(1-\theta)^{1-\theta}}\Bigr)^{\frac{1}{\epsilon}}
\end{eqnarray*}
The first two lines are automatic. For the next part, we first study the function $t(x) = \log \frac{1}{x^{x}(1-x)^{1-x}}$. We show that $t(x)$ is increasing for $x \in [0,1/2]$ and that its values lie within the interval $[0,\log 2]$. In the third line, for $r \le s/2$, we use the monotonicity of $t(x)$, whereas for $r > s/2$ and $s$ large enough, we write
\begin{eqnarray*}
\exp(t(r/s)) & = & \exp(t(\theta))+\exp(t(\theta))(\exp(t(r/s)-t(\theta))-1)\\
& = & \exp(t(\theta))+O\Bigl(\frac{1}{s}\Bigr)
\end{eqnarray*}
which follows from $|\frac{r}{s}-\theta| \ll \frac{1}{s}$ and the fact that $1 \le \exp(t(x)) \le 2$. In the fourth line, we introduce an approximation of $s$ in the exponent, along with estimates for $r$, $s-r+1$, and $s$. For the last line, we simplify the fourth one.

If the inequality $\epsilon \le \frac{\theta}{2}$ does not hold, we observe that
$$
\sqrt{\epsilon\theta^{3}}\Bigl(\frac{1}{\theta^{\theta}(1-\theta)^{1-\theta}}\Bigr)^{\frac{1}{\epsilon}} \ll 1
$$
and thus the announced result holds from Proposition \ref{prop:4}. This completes the argument for the case $\theta \le \frac{1}{2}$.

Only the case where $\theta \ge \frac{1}{2}$ remains. Let us begin with the case where $\epsilon < 1-\theta$. The previous construction can be adapted to find arbitrary large integers $n$ for which
$$
|\{d \in \mathcal{D}_{n}|\ n^{1-\theta}-0.5n^{1-\theta-\epsilon} \le d \le n^{1-\theta}\}| \gg k_{\epsilon}(\theta)
$$
so that, as we did in Proposition \ref{prop:4}, we conclude that
$$
D_{n}(n^{\theta},n^{\theta-\epsilon}) \gg k_{\epsilon}(\theta).
$$
The details of this construction are mostly the same. We choose the primes in the interval $[M-L,M]$. We let $1-\theta$ replace $\theta$ everywhere in the argument, including the definition of $r$. We define $N_{1}$, the analogue of $N_{0}$, and we note
$$
N':=\min_{\substack{v \ge 1 \\ M^{r} \le v^{1-\theta}M^{(1-\theta)s}\bigl(1-\frac{(1-\theta)s L}{M(1-sL/M)^{\theta}}\bigr)}}vN_{1}.
$$
The analogue of \eqref{ineg_1} gives $q \le N'^{1-\theta}$ and the analogue of \eqref{ineg_2} gives $q \ge N'^{1-\theta}-0.5N'^{1-\theta-\epsilon}$. Everything else is the same. Finally, in the case where $\epsilon \in [1-\theta,\theta)$, we verify that the announced result is simply $k_{\epsilon}(\theta) \gg 1$, which follows from Proposition \ref{prop:4}.

\section{A discussion around Conjecture \ref{con:1}}

It is well known that the large divisors of an integer $n$ are isolated from each other. Indeed, let $d$ and $d+h$ be two consecutive divisors of $n$ in increasing order. We observe that
$$
1 \le \frac{n}{d}-\frac{n}{d+h}=\frac{hn}{d(d+h)} \quad \mbox{which leads to} \quad h > \frac{d^{2}}{n}
$$
which is non-trivial when $d \ge \sqrt{n}$. Proposition 4.1 of \cite{pe:mr} suggests a completely different method that is most efficient when the divisors are close to $\sqrt{n}$.

Let $a, b$ be two fixed coprime positive integers and consider the function
$$
f(x):=a\frac{n}{x}+bx \quad x > 0.
$$
This simple function has all the desired properties and more. First of all, its derivative vanishes at $x_{0}=\sqrt{\frac{an}{b}}$ where $f(x)$ attains its minimum. Moreover, from its Taylor series, we have
\begin{eqnarray*}
f(x_{0}+h) & = & f(x_{0})+f''(\zeta)\frac{h^{2}}{2}\\
& = & f(x_{0})+\frac{an}{\zeta^{3}}h^{2}
\end{eqnarray*}
for some $\zeta \in [x_{0},x_{0}+h]$. Secondly, for every divisor $d$ of $n$, $f(d)$ is an integer. Also, as we restrict our attention to the divisors $d \ge x_{0}$, the value of $f(d)$ determines $d$.

We deduce that $|f(x_{0}+h)-f(x_{0})| \ge i$ requires $h \ge \sqrt{\frac{ix_{0}}{b}}$ and therefore
$$
D_{n}\bigl(x_{0},\sqrt{\frac{ix_{0}}{b}}\bigr) \le i+1.
$$

Furthermore, we note that
\begin{equation}\label{alg}
f(d)^{2}-f^{*}(d)^{2}=4abn
\end{equation}
where
$$
f^{*}(x):=a\frac{n}{x}-bx.
$$
This associated function $f^{*}$ also takes integer values at divisors $d$ of $n$. This implies that \eqref{alg} is, in fact, a nontrivial algebraic relation satisfied by $(f(d),f^{*}(d))$. Interestingly enough, this algebraic identity allows us to get constraints modulo well chosen prime powers on the values taken by $f(d)$. This is made explicit in Lemma \ref{lem:4}, and everything gets in place to apply the large sieve from Lemma \ref{lem:3}. For each $p \ge 3$, we choose the set $\mathcal{P}$ to contain $p^{v_{p}(4abn)+1}$, where $v_{p}(\cdot)$ is the usual $p$-adic valuation. We can verify that the value of the function $h(p^{\beta})$ in Lemma \ref{lem:3} is 
$$
h(p)=1+O\Bigl(\frac{1}{p}\Bigr), \quad h(p^{2})=\frac{1}{p-1}, \quad h(p^{3})=\frac{1}{2p}+O\Bigl(\frac{1}{p^{2}}\Bigr)
$$
and that it satisfies $h(p^{\beta}) \ll \frac{1}{p^{2}}$ for $\beta \ge 4$. For all $n$, we obtain
\begin{equation}\label{ine_ls}
D_{n}\bigl(x_{0},\sqrt{\frac{ix_{0}}{b}}\bigr) \le \frac{i+1+Q^{2}}{H}
\end{equation}
where $H=\sum_{q\le Q}h(q)$.

It is very difficult, but not impossible, for an integer $n$ to ensure that $H$ requires $i$ to grow large before allowing an interesting result of the form $o(i)$. One must consider integers that are products of the majority of small prime numbers raised to at least the third power. 

Nevertheless, some general statements can be deduced from \eqref{ine_ls}. For example, in the case we have $n$ respectively squarefree or cubefree, we can show that the right hand side of \eqref{ine_ls} is $\ll \frac{i}{\log i}$ and $\ll \frac{i}{\sqrt{\log i}}$ $(i \ge 2)$ respectively. Moreover, as soon as $Q \gg \log n$, there is necessarily several prime numbers that do not divide $n$, so that if $\frac{i}{\log n \log \log n} \rightarrow \infty$, we get $D_{n}\bigl(x_{0},\sqrt{\frac{ix_{0}}{b}}\bigr)=o(i)$ in full generality.

{\sc Département de mathématiques et de statistique, Université Laval, Pavillon Alexandre-Vachon, 1045 Avenue de la Médecine, Québec, QC G1V 0A6} \\
{\it E-mail address:} {\tt Patrick.Letendre.1@ulaval.ca}

\end{document}